\documentclass[a4,11pt,reqno]{amsart}

\usepackage{amsmath,amssymb,amsthm}
\usepackage[colorlinks,linkcolor={blue},citecolor={blue},urlcolor={black}]{hyperref}
\usepackage[utf8]{inputenc}
\usepackage[active]{srcltx}

\theoremstyle{plain}
\newtheorem{theorem}{Theorem}[section]
\newtheorem{corollary}[theorem]{Corollary}
\newtheorem{lemma}[theorem]{Lemma}
\newtheorem{proposition}[theorem]{Proposition}

\newtheorem{definition}[theorem]{Definition}

\newtheorem*{definition*}{Definition}

\theoremstyle{remark}

\newtheorem{example}[theorem]{Example}

\newtheorem*{claim*}{Claim}
\newtheorem*{remark*}{Remark}
\newtheorem*{example*}{Example}
\newtheorem*{notation*}{Notation}

\numberwithin{equation}{section}

\def\S{{\mathbb S}}

\def\N{{\mathbb N}}

\def\R{{\mathbb R}}

\newcommand{\dd}{{\rm d}}
\newcommand{\cP}{{\mathcal P}}

\DeclareMathOperator{\Con}{Con}

\DeclareMathOperator{\Ric}{Ric}

\DeclareMathOperator{\ent}{Ent}

\author{Matthias Erbar, Karl-Theodor Sturm}

\title{Rigidity of cones with bounded Ricci curvature}
\begin{document}

\begin{abstract}
  We show that the
  only metric measure space with the structure of an $N$-cone and
 with two-sided synthetic Ricci bounds  is the Euclidean space
  $\R^{N+1}$ for $N$ integer. This is based on a novel 
  notion of Ricci curvature upper bounds for
  metric measure spaces given in terms of the short time asymtotic of
  the heat kernel in the $L^2$-transport distance. 
  
  Moreover, we establish a beautiful
  rigidity results of independent interest which characterize the
  $N$-dimensional standard sphere $\S^N$ as the unique minimizer of 
  $$\int_X\int_X \cos d (x,y)\, m(\dd y)\,m(\dd x)$$
   among all metric measure
  spaces with dimension bounded above by $N$ and Ricci curvature bounded
  below by $N-1$.
\end{abstract}

\maketitle

\section{Introduction}
\label{sec:intro} 

\footnote{
Both authors gratefully acknowledge  support by the German Research Foundation through the Hausdorff Center for Mathematics and the Collaborative Research Center 1060. The second author also gratefully acknowledges 
 support by the European Union through the ERC-AdG ``RicciBounds''.
}The theory of synthetic curvature-dimension bounds for non-smooth
space has been very active and successful in the last decades. It was
initiated in the works of Bakry--\'Emery \cite{BE85} from the point of
view of abstract Markov semigroups and Lott--Villani \cite{LV09} and
Sturm \cite{St06} from the point of view of optimal transport and
metric measure space. Generalized lower bounds on the Ricci curvature
and upper bounds on the dimension lead to a large number of geometric
and functional inequalities and powerful control on the underlying
diffusion process. By now, many precise analytic and geometric results
for metric measure spaces under curvature-dimension bounds have been
established such as Li--Yau type estimates for heat semigroup
\cite{GM14} and splitting and rigidity results \cite{G13, Ket15} and a
clear picture of the fine structure of such spaces is emerging
\cite{MN17}.

Recently, significant progress has been made in developing more
detailed synthetic control on the Ricci curvature in a non-smooth
context. Gigli \cite{G15} and Han \cite{H14} provide a definition of
the full Ricci tensor on metric measure spaces, building upon 
a similar contruction in the context of $\Gamma$-calculus by Sturm \cite{St14}. Naber
\cite{N13} characterized two-sided bounds on the Ricci curvature in
terms of functional inequalities in the path space, see also recent
work of Cheng--Thalmaier \cite{CT16} and of Wu \cite{Wu17}.
%
%

 A drawback of the previous approaches to detailed
controls on Ricci is that they do not see curvature concentrated in
 singular sets such as the tip of a cone. One goal of the present
article is to a analyze a different concept of synthetic upper Ricci
bounds introduced recently by the first author \cite{St17} and to impose a
remarkable rigidity: the only metric measure spaces with cone
structure and with Ricci curvature bounded above and below are
Euclidean spaces $\R^N$.

We will work in the setting of RCD$^*(K',N')$ metric measure spaces, see
Section \ref{sec:prelim} for definitions and references. In this
setting an equivalent definition of lower Ricci bound $K$ is the
contraction estimate
\begin{align*}
  W_2(\hat P_t\mu,\hat P_t \nu)\leq e^{-Kt}W_2(\mu,\nu)\;,
\end{align*}
for the dual heat flow $\hat P_t$ in $L^2$-Wasserstein distance. The
central object in \cite{St17} to define upper Ricci bounds is a
reverse estimate asymptotically for short times. More precisely,
consider for a mm-space $(X,d,m)$ and $x,y\in X$
\begin{align*}
  \vartheta^+(x,y) := -\liminf_{t\to0}\frac{1}{t}\log\left(\frac{W_2(\hat P_t\delta_x,\hat P_t\delta_y)}{d(x,y)}\right)\;,
\quad \vartheta^*(x):= \limsup_{y,z\to x}\vartheta^+(y,z)\;.
\end{align*}
For smooth Riemmanian mainfolds an upper bound $\Ric\leq K$ is
equivalent to requiring $\vartheta^*(x)\leq K$ for all $x$. For an RCD
space $(X,d,m)$ we take the latter as a definition of $\Ric\leq K$,
see Section \ref{sec:upper} for more details.

\medskip

Our first  main result is the following rigidity theorem for cones.

\begin{theorem}\label{thm:rigidity-cone}
  Let $(Y,d_Y,m_Y)$ be a mm-space satisfying the curvature-dimension
  conditon RCD$^*(K',N')$ for some $K'\in R$ and $N'\in [0,\infty)$ and
  assume that it is the $N$-cone over a mm space $(X,d_X,m_X)$ for some $N\geq1$. Then:
  \begin{itemize}
  \item[(i)] either $\vartheta^+(o,y)=+\infty$ for any $y\in y$ and $o$ the
    tip of the cone,
    
  \item[(ii)] or $N$ is an integer and $(Y,d_Y,m_Y)$ is isomorphic to Euclidean
    space $R^{N+1}$ with the Euclidean distance and a multiple of the Lebesgue measure.
  \end{itemize}
  In particular, (up to isomorphism) the only $N$-cone with bounded
  Ricci curvature among all mm-spaces is the $N$-dimensional Euclidean space
  for $N$ integer.
\end{theorem}

An important ingredient to establish the rigidity of cones will be a
novel class of rigidity results characterizing the standard sphere
$\S^N$ which will be applied to characterize the base of the
cone. They are of independent interest and form the second goal of
this article.



Let $f:[0,\pi]\to \R$ be continuous and strictly increasing 
and put for a mm-space $(X,d,m)$ with $m(X)<\infty$ and $\text{diam}(X)\leq\pi$: 
\begin{align*}
  M_f(X) &:= \frac{1}{m(X)^2}\int_X\int_Xf\big(d(x,y)\big) \,dm(x) \,dm(y)\;,\\
  M_{f,N}^*&:=\Big[\int_0^\pi f\big(r\big)\sin(r)^{N-1}\dd r\Big]/\Big[\int_0^\pi\sin(r)^{N-1}\dd r\Big]\;.
\end{align*}

\begin{theorem}\label{thm:rigidity-more}
  Let $(X,d, m)$ be an RCD$^*(N-1,N)$ space with $N\geq1$,
  $\text{diam}(X)\leq \pi$. Then we have
  $M_f(X)\leq M_{f,N}^*$. Moreover, the following are equivalent:
  \begin{itemize}
    \item[(i)] $M_f(X)=M_{f,N}^*$,
    \item[(ii)] $N$ is an integer and $X$ is isomorphic to the sphere
      $\S^N$ with the round metric and a multiple of the volume measure.
  \end{itemize}
\end{theorem}

In particular, we see that for $N\in\N$ the standard sphere $\S^N$ is
the unique maximizer of the expected distance between points and of
the variance among RCD$^*(N-1,N)$ spaces, choosing $f(r)=r$ or
$f(r)=r^2$ respectively. We also establish a corresponding almost
rigidity theorem, see Theorem \ref{thm:almost rigidity}. It is easy to
see that the extremum of $M_f$ among RCD$^*(N-1,N)$ spaces is attained
also for non-integer $N$. It would be an interesting question to
characterize the extremizers in this case.

The proof of Theorem \ref{thm:rigidity-more} will rely on the maximal
diameter theorem obtained by Ketterer \cite{Ket15} which in turn stems
from Gigli's non-smooth splitting theorem \cite{G13}. In fact, we will
see that $(i)$ will imply that $m$-a.e.~point in $X$ will have a
partner at the maximal distance $\pi$. Also, the other known rigidity
results for RCD$^*(K,N)$ spaces with $K>0$, namely Ketterer's
non-smooth Obata theorem \cite{Ket15a} for spaces with extremal spectral
gap and the rigidity of spaces saturating the Levy--Gromov
isoperimetric inequality \cite{CM17}, are based on the maximal diameter
theorem.

An analogous statement (with $M_f(X)\geq M_{f,N}^*$ in the place of $M_f(X)\leq M_{f,N}^*$) holds for strictly decreasing $f$. Of particular interest is the case  $f=\cos$ which leads to $M_{\cos,N}^*=0$. 

\begin{corollary}
 Let $(X,d, m)$ be an RCD$^*(N-1,N)$ space with $N\geq1$,
  $\text{diam}(X)\leq \pi$. Then the following are equivalent:
  \begin{itemize}
    \item[(i)] $\int_X\int_X \cos d (x,y)\, m(\dd x)\,m(\dd y)\le0$
    \item[(ii)] $N$ is an integer and $X$ is isomorphic to the sphere
      $\S^N$ with the round metric and a multiple of the volume measure.
  \end{itemize}
\end{corollary}
Note  the condition $\text{diam}(X)\leq \pi$ is only requested in the case $N=1$. In the case $N>1$, it already follows from the RCD$^*(N-1,N)$-condition.

In order to obtain Theorem \ref{thm:rigidity-cone}  from this Corollary, note that the
distance on the cone $Y$ is built from the distance on $X$ via the law
of cosines. We will show that as soon as
\begin{align*}
  a:=\int_X\cos\big(d(x,y)\big)m(\dd y)>0
\end{align*}
for some point $x\in X$ we have for $p=(r,x)$ in the cone $Y$ that
\begin{align*}
W_2(\hat P_t\delta_o,\hat P_t\delta_p)^2\leq d(o,p)^2-ca\sqrt{t} +
O(t)\;,
\end{align*}
for some constant $c>0$, which implies that $\vartheta(o,p)=+\infty$.

\subsection*{Organization}
\label{sec:orga}

The article is organized as follows. In Section \ref{sec:prelim} we
recall definitions and results concerning synthetic
curvature-dimension bounds for metric measure spaces, as well as the
notion of upper bounds on the Ricci curvature considered here. The
proof of Theorem \ref{thm:rigidity-more} will be given in Section
\ref{sec:rigid-sphere} together with 
corresponding almost rigidity statements. In Section
\ref{sec:rigid-cone} we give the proof of Theorem
\ref{thm:rigidity-cone}.

\section{Preliminaries}
\label{sec:prelim}

\subsection{Synthetic Ricci bounds for metric measure spaces}
\label{sec:ric}

We briefly recall the main definitions and results concerning
synthetic curvature-dimension bounds for metric measure spaces that
will be used in the sequel.

A metric measure space (mm-space for short) is a triple $(X,d,m)$
where $(X,d)$ is a complete and seperable metric space and $m$ is a
locally finite Borel measure on $X$. In addition, we will always
assume the integrability condition
$  \int_X\exp(-c d(x_0,x)^2)\dd m(x) < \infty
$
for some $c>0$ and $x_0\in X$. We denote by $\mathcal P_2(X)$ the
space of Borel probability measures on $X$ with finite second moment
and by $W_2$ the $L^2$-Kantorovich-Wasserstein distance.

The Boltzmann entropy of $\mu\in \mathcal P(X)$ is defined by
$  \ent(\mu) = \int \rho\log \rho \dd m\
$
provided $\mu=\rho m$ is absolutely continuous w.r.t.~$m$ and
$\int\rho(\log\rho)_+\dd m<\infty$; otherwise 
$\ent(\mu)=+\infty$.
The Cheeger energy  of  $f\in L^2(X,m)$ is defined by 
\begin{align*}
 \rm{Ch}(f) =
 \liminf_{\stackrel{g\to f\ \mathrm{in}\ L^2(X,m)}{g\in Lip(X,d)}}\frac12\int|\nabla g|^2\dd
    m\;,
 \end{align*}
 where  $|\nabla g|$ denotes the local Lipschitz
constant.
A mm-space is called \emph{infinitesimally
  Hilbertian} if $\rm{Ch}$ is quadratic. In this case, $\rm{Ch}$ gives
rise to a strongly local Dirichlet form. The associated generator
$\Delta$ is called the Laplacian and the associated Markov semigroup
$(P_t)_{t\geq0}$ on $L^2(X,m)$ is called the \emph{heat flow} on
$(X,d,m)$, see  \cite{AGS14} for more
details.

For $\kappa\in \R$ and $\theta\geq 0$ define the functions
\begin{align*}
 \mathfrak{s}_\kappa(\theta) &=
  \begin{cases}
    \frac{1}{\sqrt{\kappa}} \sin(\sqrt{\kappa}\theta)\;,& \kappa>0\;,\\
\theta\;, &\kappa=0\;,\\
 \frac{1}{\sqrt{-\kappa}} \sinh(\sqrt{-\kappa}\theta)\;,& \kappa<0\;,
  \end{cases}
  \end{align*} and $\mathfrak{c}_\kappa(\theta) =\frac{d}{d\theta}\mathfrak{s}_\kappa(\theta)$.
Moreover, for $t\in[0,1]$ define the distortion coefficients
\begin{align*}
  \sigma^{(t)}_\kappa(\theta)=
  \begin{cases}
    \frac{\mathfrak{s}_\kappa(t\theta)}{\mathfrak{s}_\kappa(\theta)}\;,& \kappa\theta^2\neq 0 \text{ and } \kappa \theta^2 <\pi^2\;,\\
    t\;, & \kappa\theta^2 =0\;,\\
    +\infty\;, & \kappa\theta^2\geq \pi^2\;.
  \end{cases}
\end{align*}

\begin{definition}
i) A metric measure space  satisfies the condition
  CD$^*(K,N)$ with $K\in\R$ and $N\in [1,\infty)$ if for each pair
  $\mu_0=\rho_0m$ and $\mu_1=\rho_1m\in \cP_2(X)$ there exists an
  optimal coupling $q$ of $\mu_0,\mu_1$ and a geodesic $\mu_t=\rho_tm$
  connecting them such that \begin{align*} \int
    \rho_t^{-\frac{1}{N'}}\rho_t\dd m \geq \int
    \Big[\sigma^{(1-t)}_{K/N'}\big(d(x_0,x_1)\big)\rho_0^{-\frac{1}{N'}}+
    \sigma^{(t)}_{K/N'}\big(d(x_0,x_1)\big)\rho_1^{-\frac{1}{N'}}\Big]\dd
    q(x_0,x_1) \end{align*} holds for all $t\in[0,1]$ and all
  $N'\geq N$, see
  \cite{BS}.
  
  ii)   \label{def:RCD}
 A mm-space satisfies the condition
  RCD$^*(K,N)$ for $K\in\R$ and $N\in[1,\infty)$ if it is
  infinitesimally Hilbertian and satisfies
  CD$^*(K,N)$.
    \end{definition} 
  
 It has been shown in \cite{EKS15}
that the RCD$^*(K,N)$ condition can be formulated equivalently in
terms of Evolution Variational Inequalities.
In particular, for each $\mu_0\in \cP_2(X)$
there exists a (unique) EVI gradient flow emanating in $\mu_0$,
 denoted by
$\hat P_t\mu_0$ and called the heat flow acting on measures.
For $\mu_0=f m$ with
$f\in L^2(X,m)$ it coincides with the heat flow \cite{AGS14},
i.e.~$\hat P_t(fm)=(P_tf)m$. 
 It has
been shown (\cite[Thm.~6.1]{AGS14a}, \cite[Thm.~7.1]{AGMR15}) that the RCD
condition entails several regularization properties for $P_t$. For
instance, $P_tf(x)=\int f \dd\hat P_t\delta_x$ holds for $m$-a.e.~for
every $f\in L^2(X,m)$. This representative of $P_tf$ has the strong
Feller property, that is $x\mapsto \int f\dd \hat P_t\delta_x$ is
bounded and continuous for any bounded $f\in L^2(X,m)$.  In particular, we have the
following  estimate for the quadratic variation.

\begin{lemma}\label{lem:conv_W2}
 Let $X$ be an RCD$^*(0,N)$ space. Then we have for $\mu\in\mathcal P_2(X)$ and all $t>0$:
 \begin{align*}
 W_2(\hat P_t\mu,\mu)^2\leq 2Nt\;.
  \end{align*}
\end{lemma}

 \begin{proof} 
Choosing $K=0$, $\nu=\mu$ and $s=0$ (or more precisely, considering the limit $s\searrow0$) 
in \cite[Thm.~4.1]{EKS15} yields the claim.
\end{proof}
  
  The CD$^*(K,N)$
condition 
is a priori slightly weaker than the original condition CD$(K,N)$
given in \cite{St06}, where the coefficients
$\sigma^{(t)}_{K/N}(\theta)$ are replaced by
$\tau^{(t)}_{K/N}(\theta)=t^{1/N}\sigma^{(t)}_{K/(N-1)}(\theta)^{1-1/N}$.

 Recently, however,
  Cavaletti and Milman \cite{CM17} succeeded to show  that the condition
  CD$^*(K,N)$ is in fact equivalent to CD$(K,N)$ provided $(X,d,m)$ is non-branching -- which  in particular will be the case if it is infinitesimally Hilbertian.
  Thus in turn  RCD$^*(K,N)$ will imply the
sharp Bonnet-Myers diameter and Bishop-Gromov volume comparison
estimates, see also \cite{CS,St06} for an alternative argument.
 Given $x_0\in\rm{supp}[m]$ and $r>0$ we
denote by $v(r):=m\big(\bar B_r(x_0)\big)$ the volume of the closed
ball of radius $r$ around $x_0$ and by \begin{align*}
  s(r):=\limsup_{\delta\to0}\frac{1}{
    \delta}m\big(\overline{B_{r+\delta}(x_0)}\setminus
  B_r(x_o)\big) \end{align*} the volume of the corresponding
sphere.

\begin{proposition}\label{prop:BG} Assume that $(X,d,m)$ is non-branching and
  satisfies CD$^*(K,N)$. Then each bounded closed subset of
  $\rm{supp}[m]$ is compact and has finite volume. For each
  $x_0\in\rm{supp}[m]$ and $0<r\leq R\leq \pi\sqrt{N/(K\wedge 0)}$ we
  have \begin{align*} \frac{s(r)}{s(R)}\geq
    \left(\frac{\mathfrak{s}_{K/(N-1)}(r)}{\mathfrak{s}_{K/(N-1)}(R)}\right)^{N-1}
    \quad\text{ and }\quad \frac{v(r)}{v(R)}\geq
    \frac{\int_0^r\mathfrak{s}_{K/(N-1)}(t)^{N-1}\dd
      t}{\int_0^R\mathfrak{s}_{K/(N-1)}(t)^{N-1}\dd
      t}\;.  \end{align*} Moreover, if $K>0$ then $\rm{supp}[m]$ is
  compact and its diameter is bounded by
  $\pi\sqrt{N/K}$.
\end{proposition}

\subsection{Upper Ricci bounds} \label{sec:upper}

Here, we briefly introduce the synthetic notion of upper Ricci
curvature bounds considered in this paper. For more details we refer to \cite{St17}.
Let us mention that there also other approaches in terms the behaviour of the entropy
along Wasserstein geodesics and their relations are discussed.

Let $(X,d,m)$ be an RCD$^*(K',N')$ mm-space and let $\hat P_t$ denote
the dual heat flow acting on measures. For points $x,y\in X$ we set
\begin{align*}
  \vartheta^+(x,y) := -\liminf_{t\to0}\frac{1}{t}\log\left(\frac{W_2(\hat P_t\delta_x,\hat P_t\delta_y)}{d(x,y)}\right)\;,\qquad
 \vartheta^*(x) := \limsup_{y,z\to x}\vartheta^+(y,z)\;.
\end{align*}

It is shown in \cite[Thm.~2.10]{St17} that a lower bound
$\theta^+(x,y)\geq K$ is equivalent to the RCD$^*(K,\infty)$ condition
and in particular to the Wasserstein contraction estimate $W_2(\hat P_t\mu,\hat P_t\nu)\leq e^{-Kt}W_2(\mu,\nu)$
for all $\mu,\nu\in\cP_2(X)$ and all $t>0$.

If $(X,d,m)=(M,d,e^{-V}\text{vol})$ is a smooth weighted Riemannian manifold we have the following precise estimate on $\vartheta^+$ in terms of the Bakry--\'Emery Ricci curvature $\Ric_f=\Ric +\text{Hess}f$.

\begin{theorem}[{\cite[Thm.~3.1]{St17}}]
  For all pairs of non-conjugate points $x,y\in M$
  \begin{align*}
    \Ric_f(\gamma) \leq \vartheta^+(x,y) \leq \Ric_f(\gamma) +\sigma(\gamma) \tan^2\left(\sqrt{\sigma(\gamma)}d(x,y)/2\right)\;,
  \end{align*}
  where $\gamma=(\gamma^a)_{a\in[0,1]}$ is the (unique) constant speed geodesic connecting $x$ and $y$, 
  \begin{align*}
    \Ric_f(\gamma) = \frac{1}{d(x,y)^2}\int_0^1\Ric_f(\dot\gamma^a,\dot\gamma^a)\dd a\;,
  \end{align*}
  and $\sigma(\gamma)$ denotes the maximal modulus of the Riemannian curvature along the geodesic $\gamma$.
\end{theorem}

In particular one sees that an upper bound $\Ric_f\leq K$ for some $K\in R$ is equivalent to the estimate $\vartheta^*(x)\leq K$ for all $x\in M$. This motivates the following definition.

\begin{definition}
We say that a number $K\in\R$ is a synthetic upper Ricci bound for the mm-space $(X,d,m)$ if $\forall x\in X$
\begin{align*}
\vartheta^*(x) \leq K\;.
\end{align*}
\end{definition}

\subsection{Cones and  suspensions} \label{sec:cone-sus}
                                                                                                                                            
We recall the construction of cones for metric measure spaces.
  
\begin{definition}\label{def:cone}
For a metric measure space $(X,d_X,m_X)$ and $K\geq0$, $N\geq 1$ the $(K,N)$-cone $\Con^N_K(X)=(C,d_C,m_C)$ over $(X,d,m)$ is the defined by  $$C= \begin{cases}
       [0,\pi/\sqrt{K}]\times X/(\{0,\pi/\sqrt{K}\}\times X)\;,& K>0\;,\\
       [0,\infty)\times X/(\{0\}\times X)\;, & K=0\;,\\
     \end{cases}$$
     with
      $m_C(dr,dx)= \mathfrak{s}_K(r)^N dr\ m_X(dx)$ and $d_C$ given  for $(r,x),(s,y)\in C$ by 
    \begin{align*}
 d_C\big((r,x),(s,y)\big)
  =
    \begin{cases}
     \mathfrak{c}_K^{-1}\Big[\mathfrak{c}_K(r)\mathfrak{c}_K(s)+K\mathfrak{s}_K(r)\mathfrak{s}_K(s)\cos\big(d_X(x,y)\wedge\pi\big)\Big]\;,& K>0\;,\\
      \sqrt{r^2+s^2-2rs\cos\big(d_X(x,y)\wedge\pi\big)}\;, & K=0\;.
    \end{cases}
 \end{align*}
    
\end{definition}
We refer to the $(1,N)$-cone as the spherical suspension of $X$.

Curvature-dimension bounds for cones are intimately related to
curvature-dimension bounds for the base space. We recall the following result by Ketterer \cite{Ket15}.

\begin{theorem}\label{thm:ketterer_cones}
  Let $(X,d_X,m_X)$ be a metric measure space and let $K\geq 0$ and
  $N\geq1$. Then the $(K,N)$-cone $\Con^N_K(X)$ satisfies
  RCD$^*(KN,N+1)$ if and only if $X$ satisfies RCD$^*(N-1,N)$ and $\rm{diam}(X)\leq\pi$.
\end{theorem}

In fact, any curvature-dimension bound on the cone is sufficient to infer bounds on the base space as we will show here.
 More precisely, the following generalization holds.

\begin{theorem}\label{thm:generalized_cones}
  Let $(X,d_X,m_X)$ be a metric measure space and let $N\geq1$. Then the following statements are equivalent:
  \begin{itemize}
  \item[(i)] The $(K,N)$-cone $\Con^N_K(X)$ satisfies
  RCD$^*(K',N')$ for some $K'\in\R$ and $N'\geq N+1$.
  \item[(ii)] $X$ satisfies RCD$^*(N-1,N)$ and $\rm{diam}(X)\leq\pi$.
  \end{itemize}
  In this case $\Con^N_K(X)$ satisfies RCD$^*(KN,N+1)$.
\end{theorem}

A close inspection of the proof in \cite{Ket15} reveals that, at least
in the case of the Euclidean cone $K=0$, the arguments there already
yield that RCD$*(0,N')$ on the cone implies RCD$*(N-1,N)$ on the base
space although this is not explicitely stated. Since the argument is
quite technical and involved we sketch the main steps for the reader's
convenience and highlight the modifications. See the proof of
\cite[Thm.~1.2]{Ket15} for more details. To obtain the statement in
the case $K>0$ and under the relaxed curvature bound $K'$ we provide
additional arguments.

\begin{proof}[Proof of Thm.~\ref{thm:generalized_cones}]
  We only need to treat the implication (i)$\Rightarrow$(ii). We proceed in three steps

  {\bf Step 1:} Let us first consider the case $K=0$ and assume that
  $\text{Con}_0^N(X)$ satisfies RCD$^*(0,N')$.

 {\bf a)} Following the argument of Bacher and Sturm in
  \cite{BS} one finds that the CD$^*(0,N')$ condition for
  $C=\Con_0^N(X)$ implies that $\rm{diam} X \leq \pi$ and hence $C$
  conincedes with the warped product $[0,\infty)\times^N_{\rm{id}} X$.
  Corollary 5.15 in \cite{Ket15} yields that $X$ is infinitesimally
  Hilbertian. Prop.~5.11, Cor.~5.12 in \cite{Ket15} show that the
  Cheeger energy of $C$ coincides with the skew product of the
  Dirichlet forms on $[0,\infty)$ and $X$ and that the intrinsic
  distance of the latter coincides with $d_C$.  Moreover, with
  $I=[0,\infty)$ one has
  $C^\infty_0(I)\otimes D(\Gamma_2^X\subset D(\Gamma_2^C)$ and
  $1\otimes D_+^{b,2}(L^X)\subset D_+^{b,2}(L^C)$. Finally,
  \cite[Thm.~4.26]{Ket15} yields that the Bakry--\'Emery condition
  BE$(0,N')$ holds for the Dirichlet form on $C$.

\medskip

{\bf b)} Following the proof of \cite[Thm.~3.23]{Ket15}, using
the explicit expression of the $\Gamma_2$-operator on $C$ (see (27) in
\cite{Ket15}):
\begin{align*}
  \Gamma_2^C(u\otimes v) &= \Big((u'')^2+\frac{N}{r^2}(u')^2\Big)v^2\\
                         &\quad+ \frac{1}{r^4}u^2\Gamma^X_2(v) - \frac{N-1}{r^4}u^2\Gamma^X(v)\\
                         &\quad+ \frac{2}{r^3}uu'L^X(v)v\\
                         &\quad+\Big(\frac{2}{r^2}(u')^2-\frac{4}{r^3}uu'+\frac{2}{r^4}u^2\Big)\Gamma^X(v)\;,
\end{align*}
choosing in particular $u(r)=r$ locally, and using the Bochner
inequality with parameters $(0,N')$ in $C$ one arrives at the
following integrated estimate for $v\in D(\Gamma_2^X)$ and test
function $\phi\in D_+^{b,2}(L^X)$:
\begin{align}\nonumber
 &\int L^X \phi \Gamma^X(v)\dd m_X - \int \Gamma^X(v,L^X v)\phi\dd m_X\\\nonumber
&\geq
 (N-1)\int\Gamma^X(v)\phi\dd m_X + \frac{1}{N'} \int \big(L^Xv +Nv\big)^2 \phi \dd m_X\\\nonumber
&\quad -\int \phi\big(v^2N + 2 v L^X v\big)\dd m_X\\\nonumber
&=  (N-1)\int\Gamma^X(v)\phi\dd m_X + \frac{1}{N} \int \big(L^Xv +Nv\big)^2 \phi \dd m_X\\\nonumber
&\quad -\int \phi\big(v^2N + 2 v L^X v\big)\dd m_X 
  - \frac{N'-N}{N'N}\int \big(L^Xv +Nv\big)^2 \phi \dd m_X\\\nonumber
&= (N-1)\int\Gamma^X(v)\phi\dd m_X + \frac{1}{N} \int \big(L^Xv\big)^2 \phi \dd m_X\\\label{eq:pre-Bochner}
&\quad - \frac{N'-N}{N'N}\int \big(L^Xv +Nv\big)^2 \phi \dd m_X\;.
\end{align}

\medskip

{\bf c)} It remains to get rid of the last term in
\eqref{eq:pre-Bochner} in order to conclude that $X$ satisfies
RCD$^*(N-1,N)$. For a given point $x_0$ one could simply replace $v$ by
$v-1/N L^Xv(x_0)$ in order to make the last term vanish at $x_0$ leaving all
other terms invariant. However, since the Bochner inequality is an
integrated estimate, more care is needed.

One deduces from \eqref{eq:pre-Bochner} the gradient estimate
\begin{align*}
  \left|\nabla P^X_t v\right|^2 +\frac{c(t)}{N}\left((L^XP^X_tv)^2-\frac{N'-N}{N}P^X_t(L^Xv+Nv)^2\right)\leq P^X_t\left|\nabla v\right|^2\;.
\end{align*}
From here one can follow the argument in \cite{Ket15} further to conclude the usual gradient estimate without the extra term  $-\frac{N'-N}{N}P^X_t(L^Xv+Nv)^2$ which in turn implies the RCD$^*(N-1,N)$ condition.

\medskip

{\bf Step 2:} Let us still consider the case $K=0$ but assume that
$\text{Con}_0^N(X)$ satisfies RCD$^*(K',N')$ for some $K'\in\R$.  For
$\lambda>0$ consider the homothety $\Phi_\lambda$ of
$\text{Con}_0^N(X)$ given by $\Phi_\lambda(s,y)=(\lambda s,y)$ and
note that it maps geodesics to geodesics. Consequently, also the
induced map from $\mathcal{P}(\text{Con}_0^N(X))$ to itself acting by
push-forward maps $W_2$-geodesics to $W_2$-geodesics. Let
$(\mu_t)_{t\in[0,1]}$ be a $W_2$-geodesic and let
$\mu^\lambda_t=(\Phi_\lambda)_\#\mu_t$. By the RCD$^*(K',N')$
condition the entropy is $(K',N')$-convex along the geodesic
$\mu^\lambda_t$. One finds that
$\text{Ent}(\mu^\lambda_t)=\text{Ent}(\mu_t)-(N+1)\log\lambda$ and
that $W_2(\mu_0^\lambda,\mu_1^\lambda)=\lambda W_2(\mu_0,\mu_1)$. This
implies $(K'\lambda^2, N')$-convexity of the entropy along the
original geodesic $(\mu_t)$. Since, $(\mu_t)$ was arbitrary, letting
$\lambda\to0$ yields that $\text{Con}_0^N(X)$ satisfies RCD$^*(0,N')$
and we conclude from the first step.

\medskip

{\bf Step 3:} Let us finally consider the case $K>0$ and assume that
$\text{Con}_K^N(X)$ satisfies RCD$^*(K',N')$. The result will follow
from a simple blow-up argument. Note that the pointed rescaled spaces $(\text{Con}_{K/n^2}^N(X),o)$
converge in pointed measured Gromov-Hausdorff sense to the pointed
Euclidean cone $(\rm{Con}^N_0(X),o)$ and that the they satisfy
RCD$^*(\frac{K'}{n^2},N')$. By the stability of the conditions
CD$^*(K,N)$ and RCD$^*(K,\infty)$ under pointed measured
Gromov-Hausdorff convergence (see \cite[Thm.~29.25]{Vil09} and
\cite[Thm.~7.2, Prop.~3.33]{GMS15}) we obtain that $\rm{Con}^N_0(X)$
satisfies RCD$^*(0,N')$. From the first part of the proof we infer
that $X$ satisfies RCD$^*(N-1,N)$.
\end{proof}

\section{Rigidity of the standard sphere}
\label{sec:rigid-sphere}

Here, we give the proof of the rigidity theorem for the standard
sphere, Theorem \ref{thm:rigidity-more}. Then, we formulate an almost
rigidity statement.


\begin{proof}[Proof of Theorem \ref{thm:rigidity-more}]

  Without restriction we can assume that $m(X)=1$. We consider the
  case that $f:[0,\pi]\to\R$ is continuous and strictly
  increasing. The case of decreasing $f$ then follows by considering
  $-f$. Possibly adding a constant to $f$ we can assume without
  restriction that $f\geq0$.

  Recall that the Bishop--Gromov volume comparison Proposition
  \ref{prop:BG} asserts that for any $x\in X$:
  \begin{align}\label{eq:BG-vol}
    \frac{m_X(\big(\bar B_r(x)\big)}{m_X(\big(\bar B_R(x)\big)}
     \geq \frac{\int_0^r\sin(t)^{N-1}\dd t}{\int_0^R\sin(t)^{N-1}\dd t}=:\frac{V_r^*}{V_R^*}\;.
  \end{align}
  Fix $x\in X$ and put $g(y)=f\big(d_X(x,y)\big)$. Using that
  $m_X(X)=1$ and $\text{diam}(X)\leq \pi$ we can estimate
  \begin{align*}
    \int_Xg(y)\dd m_X(y) &= \int_0^\infty m_X\big(\{g\geq s\}\big)\dd s
                         = \int_0^{\text{diam}(X)} m_X\big(\bar B^c_{f^{-1}(s)}(x)\big)\dd s\\
&= \int_0^{\text{diam}(X)} 1-m_X\big(\bar B_{f^{-1}(s)}(x)\big)\dd s
\leq \int_0^\pi 1-\frac{V^*_{f^{-1}(s)}}{V^*_{\pi}}\dd s\\
&=\Big[\int_0^\pi f\big(r\big)\sin(r)^{N-1}\dd r\Big]/\Big[\int_0^\pi\sin(r)^{N-1}\dd r\Big] = M^*_{f,N}\;.
  \end{align*}
  Integrating over $x$ then yields the first statement.

  Let us now prove the rigidity statement. From the above argument we
  obtain also that the equality $M_f(X)=M_{f,N}^*$ implies that for
  $m_X$ a.e.~point $x$ there must exist a point $x'$ with
  $d_X(x,x')=\pi$. This implies that $N$ is an integer and that $X$ is
  isomorphic to $\S^N$ by iteratively applying the maximal diameter
  theorem \cite[Thm.~1.4]{Ket15}. Indeed, recall that the existence of points $x_1,x_1'$ with $d_X(x_1,x_1')=\pi$ implies that 
  \begin{itemize}
  \item[(a)] if $N\in[1,2)$ then either $X$ is isomorphic to the
    interval $[0,\pi]$ or $N=1$ and $X$ is isomorphic to the circle
    $\S^1$ with normalized Hausdorff measure.
  \item[(b)] if $N\geq 2$, then $X$ is isomorphic to a spherical
    suspension $\Con_1^{N-1}(Y)$ for some RCD$^*(N-2,N-1)$
    space $(Y,d_Y,m_Y)$ with $\rm{diam}Y\leq \pi$ and $m(Y)=1$.
  \end{itemize}
  In case (a), we must have $N=1$ and $X$ isomorphic to $\S^1$ since
  otherwise there would be points that do not have a partner at
  distance $\pi$. In case (b) we pick $x_2\in X$ of the form
  $x_2=(\pi/2,y_2)$ and $x_2'$ such that $d_X(x_2,x_2')=\pi$. Then we
  have $x_2'=(\pi/2,y_2')$ and $d_Y(y_2,y_2')=\pi$. We then repeat the
  previous argument inductively. After $\lfloor N\rfloor$
  steps we arrive at case (a). Thus, we conclude that $N$ is an integer and that
 $X$ is the $N-1$ fold spherical suspension over $\S^1$,
 i.e.~isomorphic to $\S^N$.
\end{proof}


We have the following almost rigidity statement.

\begin{theorem}\label{thm:almost rigidity}
  For all $\epsilon>0$ and $N\geq1$ there exists $\delta>0$ depending
  only on $\epsilon$ and $N$ such that the following holds: If $X$ is
  an RCD$^*(N-1-\delta,N+\delta)$ space with $m(X)=1$ and $ M_f(X) \leq \delta$,
  then $N$ is an integer and $d_{\rm{mGH}}(X,\S^N)\leq \epsilon$, where
  $\S^N$ is the standard $N$-sphere with normalized volume.
\end{theorem}

\begin{proof}
  Assume on the contrary that there is $\epsilon_0>0$ and a sequence
  $X_n$ of normalized RCD$^*(N-1-1/n,N+1/n)$ spaces with $M_f(X_n)\leq1/n$ and
  $d_{\rm{mGH}}(X_n,\S^N)\geq\epsilon_0$ for all $n$. By compactness of
  the class of RCD$^*(K,N)$ spaces, there exist a normalized RCD$^*(N-1,N)$
  space $X$ such that $X_n$ converges to $X$ in mGH-sense along a
  subsequence. Obviously, we still have
  $d_{\rm{mGH}}(X,\S^N)>\epsilon_0$. On the other hand, since $M_f$ is
  readily checked to be continuous w.r.t.~measured Gromov--Hausdorff
  convergence, $M_f(X) = \lim_n M_f(X_n) = 0$. But then, by the rigidity
  result Theorem \ref{thm:rigidity-more}, we have that $N$ is an
  innteger and that $X$ is isomorphic to $\S^N$, a contradiction.
\end{proof}

Let us give an alternative proof of Theorem \ref{thm:rigidity-more} in
the special case $f=\cos$ that will yield the rigidity of cones with
bounded Ricci curvature. In this case $M^*_{\cos,N}=0$. The proof is
based on a slightly different induction argument, noting the the
condition $M_{\cos}(X)=0$ directly implies $M_{\cos}(Y)=0$ if $X$ is a
suspension over $Y$.

\begin{proof}[Proof of Theorem \ref{thm:rigidity-more} for $f=\cos$]
  First note that by Bishop--Gromov volume comparison we
  have that for any $x_0\in X$:
 \begin{align*}
   \int_X\cos\big(d_X(x_0,y)\big)m_X(dy) \geq 0\;.
 \end{align*}
 Indeed, denote by $s(r)$ the volume of the sphere of radius $r$
 around $x_0$ in $X$. Since $X$ satisfies RCD$^*(N-1,N)$ the
 Bishop--Gromov volume comparison Proposition \ref{prop:BG} asserts that for all
 $0<r\leq R\leq \pi$:
 \begin{align}\label{eq:BG}
   \frac{s(r)}{s(R)} \geq
     \left(\frac{\sin(r)}{\sin(R)}\right)^{N-1}\;.
 \end{align}
 Thus we obtain
 \begin{align}\label{eq:BG-cos}
   \int_X\cos\big(d_X(x_0,y)\big)m_X(dy)
   &=\int_0^\pi\cos(r)s(r)dr\\\nonumber
   &= \int_0^{\pi/2}\cos(r)s(r)dr +\int_{\pi/2}^\pi\cos(r)s(r)dr\\\nonumber
   &=\int_0^{\pi/2}\cos(r)\big[s(r)-s(\pi-r)\big]dr\geq 0\;.
 \end{align}
 Here we have used that $\cos(r)=-\cos(\pi-r)$ and that
 $s(r)\geq s(\pi-r)$ for $r\leq \pi/2$ by \eqref{eq:BG}.
 

\medskip

The previous argument also shows that in order for $M_{\cos}(X)=0$ to
hold, for a.e.~$x\in X$ there must exist a point $x'\in X$ at maximal
distance, i.e.~with $d_X(x,x')=\pi$. The maximal diameter theorem
\cite[Thm.~1.4]{Ket15} again yields that one of the two cases (a), (b) above must hold and that in case (a)
 we must have that $N=1$ and $X$ is isomorphic to $\S^1$.
  
 In the case (b), we have from the definition of distance and measure
 in the spherical suspension:
  \begin{align*}
    0&=\int_X\int_X\cos\big(d_X(x,y)\big)m_X(\dd x)m_X(\dd y) \\
   &= \int_0^\pi \int_Y\int_0^\pi\int_Y\Big[\cos(r)\cos(s)+\sin(r)\sin(s)\cos\big(d_Y(\theta,\phi)\big)\Big]\\
 &\qquad\qquad\qquad\quad\times \sin(s)^{N-1}\sin(r)^{N-1}\dd s\;\dd r\; m_Y(\dd\theta)\;m_y(\dd \phi)\\
 &=A^2\int_Y\int_Y\cos\big(d_Y(\theta,\phi)\big)m_Y(\dd\theta)\;m_y(\dd \phi)\;,
  \end{align*}
 with
 \begin{align*}
   A=\int_0^\pi\sin(s)^N\dd s>0\;.
 \end{align*}
 This implies that also $M_{\cos}(Y)=0$ holds and we repeat
 the previous argument inductively. After $\lfloor N\rfloor$ steps we
 arrive at case (a) and conclude that $N$ is an integer and that
 $X$ is the $N-1$ fold spherical suspension over $\S^1$,
 i.e.~isomorphic to $\S^N$.
\end{proof}

\section{Rigidity of cones with bounded Ricci curvature}
\label{sec:rigid-cone}

Here, we give the proof of the rigidity result for cones with bounded
Ricci curvature, Theorem \ref{thm:rigidity-cone}.

A crucial ingredient in the proof will be the relation between the
vanishing of the integral
\begin{align}\label{eq:cosint0}
  \int_X\int_X\cos\big(d(x,y)\big)m(\dd x)m(\dd y)=0
\end{align}
and the asymptotic behaviour as $t\to0$ of
$W_2\big(\hat P_t \delta_o,\hat P_t\delta_q\big)$ for the vertex $o$
of the cone and any other point $q$. We will first prove the following
pointwise equivalence which is somewhat stronger than what is needed
in the proof of Theorem \ref{thm:rigidity-cone}.

\begin{proposition}\label{prop:cosint-asymptotic}
  Let $(X,d_X,m_X)$ be an RCD$^*(N-1,N)$ space with $N\geq1$ and
  $\rm{diam}(X)\leq \pi$. Then for any $p_0=(r_0,x_0)\in \Con_0^N(X)$
  and $o$ the vertex one of the following statements holds:
  \begin{itemize}
  \item[(i)] $\int_X\cos\big(d_X(x_0,y)\big)m_X(dy) =0$ and $\vartheta^+(o,p_0)=0$.
   \item[(ii)] $\int_X\cos\big(d_X(x_0,y)\big)m_X(dy)>0$ and $\vartheta^+(o,p_0)=+\infty$.
  \end{itemize}
\end{proposition}

\begin{proof}
 {\bf Step 1:} Let us fix $p_0=(r_0,x_0)\in \Con_0^N(X)$.  Recall from \eqref{eq:BG-cos}  that
 \begin{align*}
   \int_X\cos\big(d_X(x_0,y)\big)m_X(dy) \geq 0\;.
 \end{align*}
  
 {\bf Step 2:} Let us assume first that
 $a:=\int_X\cos\big(d_X(x_0,y)\big)m_X(dy)>0$. We claim that as
 $t\to0$ we have
 \begin{align}\label{eq:OsqrtT}
   W_2(\hat P_t\delta_o,\hat P_t\delta_{p_0})^2\leq d_c(o,p_0)^2 - O\big(\sqrt{t}\big)\;,
 \end{align}
 which immediately implies that
 $-\partial_t^-\big\vert_{t=0} \log W_2\big(\hat P_t \delta_o,\hat
 P_t\delta_{p_0}\big) =+\infty$.
 To this end, denote by $\nu_p^t=\hat P_t\delta_p$ the heat kernel
 measure at time $t$ centered at $p=(r,x)\in\Con_0^N(X)$. Denote by
  $\bar\nu^t_{r}$ its marginal in the radial component. Further we
  consider the desintegration $\nu^t_{p,s}\in P(X)$ of $\nu^t_p$ after $\bar\nu^t_{r}$, i.e.~
  \begin{align*}
    \nu^t_p(ds,dy) = \bar\nu^t_{r}(ds)\nu^t_{p,r}(dy)\;.
  \end{align*}
 Lemma \ref{lem:Bessel} gives that for $p=o$ we have that $\nu^t_{o,r}=m_X$ is the uniform
  distribution on $X$.
  Let
now, $\pi=\nu^t_{p_0}\otimes\nu^t_o$ be the product coupling. We
obtain
  \begin{align*}
    W_2\big(\hat P_t \delta_o,\hat P_t\delta_{p_0}\big)^2
    &\leq \int d_C^2d\pi\\
    &= \int \big[r^2+s^2-2rs\cos\big(d_X(x,y)\big)\big]\nu^t_{p_0}(dr,dx)\nu^t_o(ds,dy)\\
 &= \int r^2 \bar\nu^t_{r_0}(dr) +\int s^2\bar\nu^t_{0}(ds)\\
&\qquad-2\int rs \cos\big(d_X(x,y)\big)\nu^t_{p_0}(dr,dx)d\bar\nu^t_{o}(ds)m_X(dy)\\
 &=\int r^2 \bar\nu^t_{r_0}(dr) +\int s^2\bar\nu^t_{0}(ds) -2 \int f d\nu^t_{p_0}\left(\int s\bar\nu^t_{0}(ds)\right)\;,
\end{align*}
where we have set $f(r,x)=r
\int_X\cos\big(d_X(x,y)\big)m_X(dy)$.
Obviously, $f$ is a continuous function on $\Con_0^N(X)$ with at most
linear growth. Since $W_2(\nu^t_{p_0},\delta_{p_0})\to0$ as $t\to 0$ by Lemma \ref{lem:conv_W2} we have that
\begin{align*}
  \int f d\nu^t_{p_0} &= f(p_0) +o(1) = r_0\int_X\cos\big(d_X(x_0,y)\big)m_X(dy) +o(1) = r_0 a +o(1)\;.
\end{align*}
Thus, using the moment estimates from Lemma \ref{lem:Bessel} we obtain
\begin{align*}
  W_2\big(\hat P_t \delta_o,\hat P_t\delta_{p_0}\big)^2
&\leq r_0^2+2Ct -2c\sqrt{t}\big(r_0a+o(1)\big)\;.
\end{align*}
for suitable constants $C,c>0$. This proves \eqref{eq:OsqrtT}.

\medskip
 
{\bf Step 3:} Let us now assume that
$\int_X\cos\big(d_X(x_0,y)\big)m_X(dy)=0$. We claim that 
\begin{align}\label{eq:OT}
   W_2(\hat P_t\delta_o,\hat P_t\delta_{p_0})\geq d_C(o,p_0) + O\big(t\big)\;,
\end{align}
which immediately implies that
$-\partial_t^-\big\vert_{t=0} \log W_2\big(\hat P_t \delta_o,\hat
P_t\delta_{p_0}\big) \leq 0$.
To this end, consider the function $\phi:\Con_0^N(X)\to\R$ given by
$\phi(s,y)=s\cos\big(d_X(x_0,y)\big)$. By Lemma \ref{lem:1Lip}, $\phi$
is $1$-Lipschitz w.r.t.~the cone distance. Hence, by
Kantorovich--Rubinstein duality, we obtain
\begin{align*}
  W_2(\hat P_t\delta_o,\hat P_t\delta_{p_0})
  &\geq W_1(\hat P_t\delta_o,\hat P_t\delta_{p_0})
    \geq \int_{\Con_0^N(X)} \phi\;d(\nu^t_{p_0}-\nu^t_o)
    = \int \phi\;d\nu^t_{p_0} =:g(t)\;.
\end{align*}
Using the definition of the cone distance we write
\begin{align*}
  2r_0g(t)
 &= -\int d_C(p_0,\cdot)^2\;d\nu^t_{p_0} + r_0^2 + \int s^2\;\nu^t_{p_0}(ds,dy)\\
&=  -\int d_C(p_0,\cdot)^2\;d\nu^t_{p_0} + 2r_0^2 +Nt\;.
\end{align*}
By Lemma \ref{lem:conv_W2} we have as $t\to0$ that
\begin{align*}
  \int d_C(p_0,\cdot)^2\;d\nu^t_{p_0} = O(t)\;.
\end{align*}
Thus, $g(t)=r_0+O(t)$ which yields \eqref{eq:OT}.

\medskip

{\bf Step 4:} Finally, recall that the RCD$^*(0,N+1)$ property of
$\Con_0^N(X)$ implies the contraction estimate
\begin{align*}
    W_2(\hat P_t\delta_p,\hat P_t\delta_q)\leq d_C(p,q)\quad\forall p,q\in C\;,
\end{align*}
which implies that $-\partial_t^-\big\vert_{t=0} \log W_2\big(\hat P_t \delta_o,\hat
P_t\delta_{p_0}\big) \geq 0$.
\end{proof}

\begin{proof}[Proof of Thm.~\ref{thm:rigidity-cone}]
  {\bf (i) $\Rightarrow$ (ii):} By Theorem
  \ref{thm:generalized_cones}, $X$ satisfies
  RCD$^*(N-1,N)$. Moreover, the assumption that $\Ric Y<K''$ implies that 
 there exists $q$ such that 
 \begin{align*}
   -\partial_t^-\big\vert_{t=0} \log W_2\big(\hat P_t \delta_o,\hat P_t\delta_{p_0}\big) <+\infty\;.
 \end{align*}
  Thus, Proposition \ref{prop:cosint-asymptotic} yields that 
  \begin{align*}
    \int_X\cos\big(d_X(x,y)\big)m_X(\dd y)=0 \qquad\forall x\in X\;, 
  \end{align*}
  and in particular \eqref{eq:cosint0} holds. Theorem
  \ref{thm:rigidity-more} with $f=\cos$ yields that $N$ is an integer and $X$ is
  isomorphic to $\S^N$ with the round metric and a multiple of the volume measure. Hence $X$ is isomorphic to $\R^{N+1}$ with Euclidean distance and a multiple of the Lebesgue measure.

\medskip

{\bf (ii)$\Rightarrow$ (i):} If $Y$ is isomorphic to $\R^{N+1}$, it
satisfies RCD$^*(0,N+1)$ and it is isomorphic to the $N$-cone
$\Con_0^N(\S^N)$. Moreover, we have that
  \begin{align*}
W_2\big(\hat P_t \delta_p,\hat P_t\delta_q\big)=d_Y(p,q)
  \end{align*}
  for all $p,q$ and hence (i) follows.
\end{proof}

\begin{lemma}\label{lem:Bessel}
  Let $\nu^t_p=\hat P_t\delta_p$ for $p=(r,x)$ and denote by
  $\bar\nu^t_{r}$ its marginal in the radial component. Further, let $\nu^t_{p,s}\in P(X)$ be the desintegration of $\nu^t_p$ after $\bar\nu^t_{r}$, i.e.~
  \begin{align*}
    \nu^t_p(ds,dy) = \bar\nu^t_{r}(ds)\nu^t_{p,r}(dy)\;.
  \end{align*}
  Then there are constants $c,C>0$ such that 
 \begin{align*}
   \int s^2 \dd\bar\nu^t_r(s) &\leq r^2 + C t\;,\\
   \int s \dd\bar\nu^t_o(s) &= c\sqrt{t} \;.
  \end{align*}
  Mor precisely, the constants are given by $C=\lceil N\rceil$ and
  $c=\int s \dd\bar\nu^1_o(s)$. Furthermore, for $p=o$ we have that
  $\nu^t_{o,r}=m_X$ is the uniform distribution on $X$.
\end{lemma}

\begin{proof}
  First note that $\bar\nu^t_r$ coincides with the heat flow in the
  RCD$^*(0,N+1)$ space $B=\big([0,\infty),|\cdot|,r^N\dd r\big)$. To see
  this, recall that $\nu^t_p$ satisfies the Evolution Variational
  Inequality EVI$_{(0,\infty)}$ on $\text{Con}_N(X)$. One can check from this that the projection $\bar\nu^t_r$ satisfies EVI$_{(0,\infty)}$
  on $B$ and conlude by recalling that the heat flow is the unique
  solution to EVI.

  Now, the Cheeger energy on $B$ is given by the
  closure of the quadratic form $\mathcal{E}^B$ on
  $C^\infty_c(0,\infty)$ given by
  \begin{align*}
    \mathcal{E}^B(u)=\int_0^\infty |u'|^2(r)r^N\dd r\;,
  \end{align*}
  see e.g.~\cite[Sec.~2.3]{Ket15}. It follows that $\bar\nu^t_r$
  coincides with the law of the $N$-dimensional Bessel process started
  from $r$. To obtain the first estimate, one uses that the second
  moment of the $N$-dimensional Bessel process is controlled from above by the one of the $M$-dimensional Bessel process if $N\leq M$ and that
  for $N\in\N$ the $N$-dim. Bessel process is obtained as the absolute
  value of a $N$-dimensional Brownian motion.

  To obtain the second statement, one employs the scaling property of
  the Bessel process $(X_t)$ starting form $0$. Namely, the law of
  $X_t$ coincides with the image of the law of $X_1$ under the
  homothety $r\mapsto\sqrt{t}r$.
  
  Finally, the last statement is obtained by noting that the measures
  $\mu^t$ on $\text{Con}_N(X)$ given by
  $\mu^t(\dd s,\dd y):=\bar\nu^t_{0}(\dd s)m_X(\dd y)$ satisfy
  EVI$_{(0,\infty)}$ which follows from the correspoding property of
  $\bar\nu^t_0$ on $B$. Thus by uniqueness $\mu^t$ conicides with
  $\nu^t_o$.
\end{proof}

\begin{lemma}\label{lem:1Lip}
  Let $(X,d_X,m_X)$ be a metric measure space with
  $\rm{diam}(X)\leq \pi$ and $x\in X$. Then the function $\phi:\Con_0^N(X)\to\R$ given by
   \begin{align*}
    (s,y)\mapsto s\cos\big(d_X(x,y)\big)
  \end{align*}
  is $1$-Lipschitz w.r.t.~the cone distance.
\end{lemma}

\begin{proof}
  Let $(s,y),(s',y')\in C=\Con_0^N(X)$ and set $\alpha =d_X(x,y)$,
  $\alpha'=d_X(x,y')$ and $\beta=d_X(y,y')$. Note that
  $\alpha,\alpha',\beta\leq\pi$ and $\beta\geq|\alpha-\alpha'|$. Let
  $p,p'\in \R^2$ be points at angle $\alpha$ and $\alpha'$ with the
  first coordinate axis respectively and $||p||=s$, $||p'||=s'$. Now we have that 
  \begin{align*}
    d_C\big((s,y),(s',y')\big)^2 &= s^2+(s')^2-2ss'\cos\beta \geq s^2 +(s')^2-2ss'\cos|\alpha-\alpha'|\\
                               &=||p-p'||^2\;.
  \end{align*}
  On the other hand, we find that
  \begin{align*}
    |\phi(s,y)-\phi(s',y')| = |s\cos\alpha -s'\cos\alpha'| = ||q-q'||\leq ||p-p'||\;,
  \end{align*}
  where $q$ and $q'$ are the projections of $p$ and $p'$ respectively
  onto the first coordinate axis.
\end{proof}

\begin{example}
  Consider the special case $X=\S^2(1/\sqrt{3})\times \S^2(1/\sqrt{3})$
  equipped with the Cartesian product of the standard Riemannian
  distances on the spheres $\S^2(1/\sqrt{3})$ with radius $1/\sqrt{3}$
  and the normalized product measure, which is an RCD$^*(3,4)$
  space. 
  Hence, the $4$-cone over $\S^2(1/\sqrt{3})\times \S^2(1/\sqrt{3})$ is
  an RCD$^*(0,5)$ space with Ricci curvature $+\infty$ at the tip.
\end{example}


%

\end{document}